\documentclass[a4paper,10pt]{amsart}
\usepackage{enumerate, amsmath, amsfonts, amssymb, amsthm, mathtools, thmtools, wasysym, graphics, graphicx, xcolor, frcursive,xparse,comment,ytableau,bbm}
\usepackage{ mathrsfs }
\usepackage{dynkin-diagrams}
\usepackage[all]{xy}
\usepackage{hyperref}
\usepackage{csquotes}

\usepackage{url, hypcap}
\hypersetup{colorlinks=true, citecolor=darkblue, linkcolor=darkblue}

\usepackage{tikz}
\usetikzlibrary{calc,through,backgrounds,shapes,matrix}

\newcommand{\convexpath}[2]{
[
    create hullnodes/.code={
        \global\edef\namelist{#1}
        \foreach [count=\counter] \nodename in \namelist {
            \global\edef\numberofnodes{\counter}
            \node at (\nodename) [draw=none,name=hullnode\counter] {};
        }
        \node at (hullnode\numberofnodes) [name=hullnode0,draw=none] {};
        \pgfmathtruncatemacro\lastnumber{\numberofnodes+1}
        \node at (hullnode1) [name=hullnode\lastnumber,draw=none] {};
    },
    create hullnodes
]
($(hullnode1)!#2!-90:(hullnode0)$)
\foreach [
    evaluate=\currentnode as \previousnode using \currentnode-1,
    evaluate=\currentnode as \nextnode using \currentnode+1
    ] \currentnode in {1,...,\numberofnodes} {
  let
    \p1 = ($(hullnode\currentnode)!#2!-90:(hullnode\previousnode)$),
    \p2 = ($(hullnode\currentnode)!#2!90:(hullnode\nextnode)$),
    \p3 = ($(\p1) - (hullnode\currentnode)$),
    \n1 = {atan2(\y3,\x3)},
    \p4 = ($(\p2) - (hullnode\currentnode)$),
    \n2 = {atan2(\y4,\x4)},
    \n{delta} = {-Mod(\n1-\n2,360)}
  in
    {-- (\p1) arc[start angle=\n1, delta angle=\n{delta}, radius=#2] -- (\p2)}
}
-- cycle
}

\definecolor{darkblue}{rgb}{0.0,0,0.7} 
\newcommand{\darkblue}{\color{darkblue}} 
\definecolor{darkred}{rgb}{0.7,0,0} 
\definecolor{lightgrey}{rgb}{0.7,0.7,0.7} 

\usepackage{graphicx}

\def\S{\mathrm{Stir}}
\def\Ell{\mathscr{L}}
\newcommand{\inv}{{\sf inv}}
\newcommand{\gc}{d}
\newcommand{\hgt}{{\sf ht}}
\newcommand{\Hgt}{{\sf Ht}}
\newtheorem{theorem}{Theorem}[section]
\newtheorem{proposition}[theorem]{Proposition}
\newtheorem{fact}[theorem]{Fact}
\newtheorem{corollary}[theorem]{Corollary}

\theoremstyle{definition}
\newtheorem{definition}[theorem]{Definition}
\newtheorem{example}[theorem]{Example}

\newtheorem{remark}[theorem]{Remark}

\usepackage[nameinlink]{cleveref}
\usepackage[all]{xy}
\usepackage[T1]{fontenc}
\crefformat{footnote}{#2\footnotemark[#1]#3}
\crefformat{conjecture}{Conjecture~#2#1#3}
\crefformat{fact}{Fact~#2#1#3}

\DeclareFontFamily{U}{rcjhbltx}{}
\DeclareFontShape{U}{rcjhbltx}{m}{n}{<->rcjhbltx}{}
\DeclareSymbolFont{hebrewletters}{U}{rcjhbltx}{m}{n}
\DeclareMathSymbol{\lamed}{\mathord}{hebrewletters}{108}
\usepackage[colorinlistoftodos]{todonotes}

\newcommand{\defn}[1]{\emph{\darkblue #1}}
\usetikzlibrary{math}
\usepackage{xifthen}
\usepackage{xstring}
\newcommand\IfStringInList[2]{\IfSubStr{,#2,}{,#1,}}

\title[Reflexponents]
  {Reflexponents}

\author[N.~Williams]{Nathan Williams}
\address[N.~Williams]{University of Texas at Dallas}
\email{nathan.f.williams@gmail.com}

\date{\today}

\begin{document}

\begin{abstract}
Certain classical generating functions for elements of reflection groups can be expressed using fundamental invariants called \emph{exponents}.
We give new analogues of such generating functions that accommodate orbits of reflecting hyperplanes using similar invariants we call \emph{reflexponents}.  Our verifications are case-by-case. 
\end{abstract}


\maketitle

\section{Introduction}

In his address at the 1950 International Congress of Mathematicians~\cite{chevalley1950betti,borel1955betti}, 
Chevalley gave the Poincar\'e series of the exceptional simple Lie groups as a product $\prod_{i=1}^n (1+q^{2e_i+1})$.  From the audience, Coxeter recognized these \defn{exponents} \[e_1 \leq e_2 \leq \cdots \leq e_n\] from his earlier computations of the eigenvalues of a product of the simple reflections of the corresponding Weyl group~\cite{coxeter1973regular,coxeter1951product}.
 Such numerology has led to many investigations~\cite{shephard1954finite,shephard1956some,coleman1958betti,steinberg1959finite,solomon1963invariants,kostant2009principal}.


\subsection{Exponents in complex reflection groups}
Let $G$ be a finite irreducible complex reflection group with reflections $\mathcal{R}$ and reflecting hyperplanes $\mathcal{H}$, acting in the reflection representation on the $n$-dimensional complex vector space $V$.  The \defn{exponents} $e_1\leq \cdots \leq e_n$ of $G$ may be defined to be the \defn{fake degrees} of $V$---the degrees in which $V$ occurs in the coinvariant ring of $G$.  The \defn{coexponents} $e_1^* \leq \cdots \leq e^*_n$ are defined to be the fake degrees of $V^*$.  For $\rho$ a representation of $G$ and $g \in G$, write \[\mathcal{M}_\rho(g):=\mathrm{dim}\Big(\mathrm{im}\big(\rho(g)-\mathbbm{1}\big)\Big).\]  The following is well-known~\cite{shephard1954finite,solomon1963invariants}.



\begin{theorem}
\label{thm:solomon}
For $G$ a finite irreducible complex reflection group,
\begin{align*}
\sum_{g \in G} x^{\mathcal{M}_V(g)} = \prod_{i=1}^n \left(1+e_i x\right) \text{ and }
\sum_{g \in G} \mathrm{det}(g) x^{\mathcal{M}_V(g)} = \prod_{i=1}^n \left(1-e_i^* x\right).
\end{align*}
\end{theorem}

Our first result gives an analogue of~\Cref{thm:solomon} that incorporates different $G$-orbits of reflecting hyperplanes, using new invariants similar to the (co)exponents.



The reflecting hyperplanes $\mathcal{H}$ of the group $G$ are broken into (at most three) $G$-orbits $\mathcal{H}/G=\{\mathcal{H}_\epsilon\}_{\epsilon \in \{s,t,u\}}$.  Choose one such orbit $\mathcal{H}_\epsilon$, and let $\mathcal{R}_\epsilon$ be the associated set of reflections.  With the exception of two cases addressed in~\Cref{sec:extensions} ($\epsilon = t$ for both $G_{13}$ and $G(ab,b,n)$ with $a,b>1$ and $n>2$), we find in \Cref{sec:num} a particular irreducible representation $V_\epsilon$ of $G$ that restricts to the reflection representation of a parabolic subgroup of $G$ supported on $\mathcal{R}_\epsilon$.  In these cases, we call the orbit $\mathcal{H}_\epsilon$  \defn{well-restricted} (see~\Cref{def:well_restricted}); the \defn{reflexponents}  $\{\epsilon_i\}_{i=1}^{n_\epsilon}$ for the orbit $\mathcal{H}_\epsilon$ are the fake degrees of $V_\epsilon$, while the \defn{co-reflexponents} $\{\epsilon_i^*\}_{i=1}^{n_\epsilon}$ are the fake degrees of $V_\epsilon^*$.   

\begin{theorem}
\label{thm:solomon2}
For $G$ a finite irreducible complex reflection group and $\mathcal{H}_\epsilon$ a well-restricted orbit of reflecting hyperplanes, there is a reindexing of the (co)exponents by $1,\ldots,n$ (with undefined (co)exponents taken to be zero), such that 
\begin{align*}
 \sum_{g \in G} \left(x/y\right)^{\mathcal{M}_{V_\epsilon}(g)}y^{\mathcal{M}_V(g)} &= \prod_{i=1}^n \Big(1+\epsilon_{i} x+(e_i-\epsilon_{i})y\Big),\\
 \sum_{g \in G} \mathrm{det}(g)\left(x/y\right)^{\mathcal{M}_{V_\epsilon}(g)}y^{\mathcal{M}_V(g)} &= \prod_{i=1}^n \Big(1-\epsilon^*_{i} x-(e_i^*-\epsilon_i^*)y\Big).
\end{align*}
\end{theorem}

\begin{example}
The dihedral group $G(2,1,2)=\langle s,t\rangle$ has four reflecting hyperplanes, four reflections $s,t,u=sts,v=tst$, and two orbits of reflecting hyperplanes.  Its eight elements are listed in~\Cref{fig:g212}.  Both of the orbits $\mathcal{H}_s$ and $\mathcal{H}_t$ are well-restricted, with corresponding one-dimensional representations \[V_{s}: s \mapsto -1, t \mapsto 1 \text{ and } V_{t}: s \mapsto 1, t \mapsto -1.\] 

As $G(2,1,2)$ has two (co)exponents $e_1=e_1^*=1$ and $e_2=e_2^*=3$, and a single (co)reflexponent $s_1=s^*_1=2=t_1=t_1^*$, we confirm that for $\epsilon \in \{s,t\}$:
\begin{align*}
 \sum_{g \in G(2,1,2)} \left(x/y\right)^{\mathcal{M}_{V_\epsilon}(g)}y^{\mathcal{M}_V(g)} = 1{+}2y{+}2x{+}2xy{+}y^2 &=(1+2x+y)(1+y),\\
 \sum_{g \in G(2,1,2)} \mathrm{det}(g)\left(x/y\right)^{\mathcal{M}_{V_\epsilon}(g)}y^{\mathcal{M}_V(g)} = 1{-}2y{-}2x{+}2xy{+}y^2 &=(1-2x-y)(1-y).
\end{align*}
\label{ex:g212}
\end{example}

\begin{figure}[htbp]
\[\begin{array}{c|c|c|c}
g &  \mathcal{M}_V(g) & \mathcal{M}_{V_s}(g) &  \mathcal{M}_{V_{t}}(g)\\\hline
e & 0 &  0 & 0  \\
sts & 1 & 0 & 1 \\
s  & 1  & 1 & 0 \\
st=us=vu =tv & 2& 1 & 1 \\ 
t & 1 &  0 & 1 \\
ut=sv=vs=tu&  2& 0 & 0 \\ 
ts=su=uv=vt & 2 & 1& 1  \\ 
v  & 1  & 1 & 0 \\ 
\end{array}\]
\caption{The eight elements of $G(2,1,2)$ described by reduced word in reflections, with the statistics $\mathcal{M}_V$, $\mathrm{det}$, $\mathcal{M}_{V_s}$, $\mathcal{M}_{V_t}$ used in~\Cref{ex:g212}.}
\label{fig:g212}
\end{figure}


In~\Cref{sec:extensions} we discuss generalizations of~\Cref{thm:solomon2} to non-well-restricted orbits.

\subsection{Exponents in Weyl groups}
Let now $\Phi$ be an irreducible crystallographic root system of rank $n$.  The \defn{exponents} $e_1 \leq e_2 \leq \cdots \leq e_n$ of $\Phi$ may be computed as the partition dual to the heights of the positive roots $\Phi^+$~\cite{kostant2009principal}.   When $G$ is a Weyl group, the exponents match the previous definition using fake degrees (and are equal to the coexponents).   The \defn{degrees} of $\Phi$ are then defined to be $d_i:=e_i{+}1$ for $1 \leq i \leq n$.

Let $W$ and $\widetilde{W}$ be the Weyl and affine Weyl groups associated to an irreducible crystallographic root system $\Phi$ of rank $n$.  For $w \in W$ or $w \in \widetilde{W}$, write $\ell(w)$ for the length of a reduced word for $w$ in simple reflections.  The following is well-known~\cite{humphreys1992reflection}.

\begin{theorem}
\label{thm:unweighted}
For an irreducible crystallographic root system $\Phi$,
\begin{align*}
\sum_{w \in W} q^{\ell(w)} = \prod_{i=1}^n \left(\frac{q^{d_i}-1}{q-1}\right) \text{ and }
\sum_{w \in \widetilde{W}} q^{\ell(w)} = (1-q)^{-n} \prod_{i=1}^n \left(\frac{q^{d_i}-1}{q^{e_i}-1}\right).
\end{align*}
\end{theorem}

Following~\cite{macdonald1972poincare}, our second result gives a weighted analogue of \Cref{thm:unweighted} that incorporates different lengths of roots.  Normalize $\Phi$ so that the short roots have length $1$ and the long roots length $r$.  Define the \defn{short exponents} $\epsilon_1 \leq \epsilon_2 \leq \cdots \leq \epsilon_{m}$ to be the partition dual to the heights of the short roots in the \emph{dual} root system, and define $\delta_i:=\epsilon_i{+}1$ to be the \defn{short degrees}.  We verify in~\Cref{thm:small_agree} that the reflexponents for the hyperplane orbit corresponding to the long roots of a Weyl group $W$ match the small exponents of $\Phi$.

Let $\inv(w)$ be the \defn{inversion set} of $w$, so that $\ell(w)=|\inv(w)|$.  We incorporate root lengths into $\ell(w)$ with the statistic $\Ell(w) :=\sum_{\alpha \in \inv(w)} \|\alpha\|^2.$ 


\begin{theorem}
\label{thm:weighted}
For an irreducible crystallographic root system $\Phi$,
\begin{align*} \sum_{w \in W} q^{\Ell_S(w)} &= \prod_{i=1}^{n_\epsilon}\left(\frac{1+q^{\delta_i}+\cdots+q^{\delta_i(r-1)}}{1+q+\cdots+q^{r-1}}\right)\prod_{i=1}^n \left(\frac{q^{d_i}-1}{q-1}\right),
\text{ and} \\
\sum_{w \in \widetilde{W}} q^{\Ell_S(w)} &= (1-q)^{-n}\prod_{i=1}^{n_\epsilon} \left(\frac{q^{\epsilon_i}-1}{q^{\epsilon_i r}-1}\right) \prod_{i=1}^n \left(\frac{q^{d_i}-1}{q^{e_i}-1}\right).
\end{align*} 
\end{theorem}

\begin{example}
The Weyl group of type $C_2$ has eight elements, listed in~\Cref{fig:g212}.  Here $\alpha_2$ and $\alpha_1+\alpha_2$ are short (of length 1), while $\alpha_1$ and $\alpha_1+2\alpha_2$ are long (of length 2), so that $r=2$.   As that $n=2$, $d_1=2$, and $d_2=4$, while $n_\epsilon=1$, $\epsilon_1=2$ and $\delta_1=3$, we confirm using \Cref{fig:c2} and \Cref{fig:c22} that
\begin{align*}
\sum_{w \in W} q^{\Ell_S(w)}=q^6 {+} q^5 {+} q^4 {+} 2q^3 {+} q^2 {+} q {+} 1 &= \left(\frac{1+q^3}{1+q}\right) \left(\frac{1-q^2}{1-q}\right)\left(\frac{1-q^4}{1-q}\right),
\\
\frac{\sum_{w \in \widetilde{W}} q^{\Ell_S(w)}}{\sum_{w \in W} q^{\Ell_S(w)}}=1{+}q{+}q^3{+}2q^4{+}q^5{+}q^6{+}\scalebox{0.5}{$\cdots$}&=\left(\frac{q^2 - 1}{q^4 - 1}\right)\left(\frac{1}{q - 1}\right)\left(\frac{1}{q^3 - 1}\right).
\end{align*}
\label{ex:c2}
\end{example}

\begin{figure}[htbp]
\[\begin{array}{l|c|r|r}
w & \inv(w) & \ell_S(w) &  \Ell_S(w) \\\hline
e &\emptyset & 0 & 0\\
sts & \alpha_1, \alpha_1 + \alpha_2, \alpha_1 + 2\alpha_2 & 3 & 5 \\
s & \alpha_1 & 1  & 2\\
st & \alpha_2, \alpha_1 + 2\alpha_2 & 2 & 3\\
t & \alpha_2 & 1 &  1\\
stst  &\alpha_1, \alpha_1 + 2\alpha_2, \alpha_1 + \alpha_2, \alpha_2 & 4 &6\\
ts & \alpha_1, \alpha_1 + \alpha_2 &2 & 3\\
tst  & \alpha_2, \alpha_1 + 2\alpha_2, \alpha_1 + \alpha_2  & 3 &4\\
\end{array}\]
\caption{The elements of the Weyl group of type $C_2$, described by reduced word in simple reflections and inversion set, with the statistics $\ell_S$ and $\Ell_S$.}
\label{fig:c2}
\end{figure}

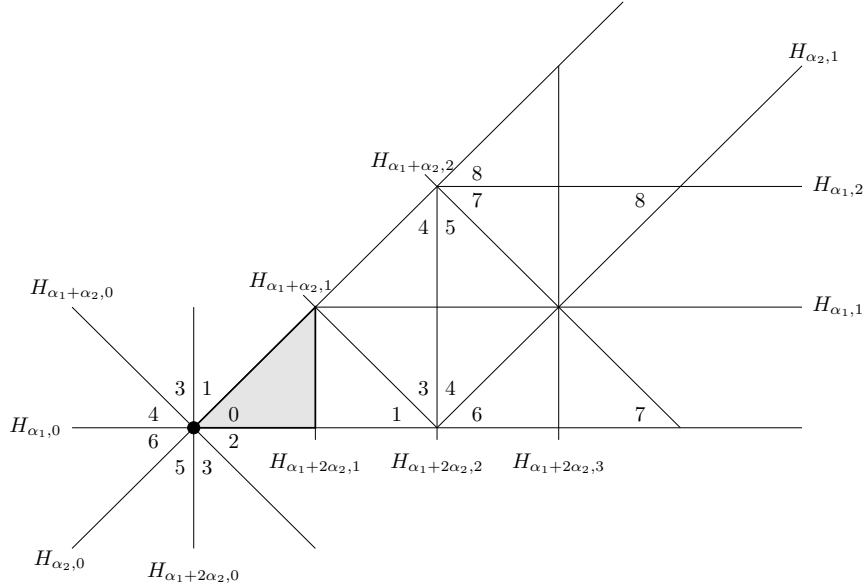
\begin{figure}[htbp]
\begin{center}
\scalebox{0.8}{
\begin{tikzpicture}[scale=2]

\draw[-,black,thin] (-1,0) -- (5,0);
\node at (-1.3,0) (010) {$H_{\alpha_1,0}$};
\draw[-,black,thin] (1,1) -- (5,1);
\node at (5.3,1) (021) {$H_{\alpha_1,1}$};
\draw[-,black,thin] (2,2) -- (5,2);
\node at (5.3,2) (022) {$H_{\alpha_1,2}$};

\draw[-,black,thin] (-1,-1) -- (3.53,3.53);
\node at (-1.1,-1.1) (100) {$H_{\alpha_2,0}$};
\draw[-,black,thin] (2,0) -- (5,3);
\node at (5.1,3.1) (101) {$H_{\alpha_2,1}$};

\draw[-,black,thin] (1,1) -- (1,-.1);
\node at (1,-.3) (211) {$H_{\alpha_1+2\alpha_2,1}$};
\draw[-,black,thin] (2,2) -- (2,-.1);
\node at (2,-.3) (212) {$H_{\alpha_1+2\alpha_2,2}$};
\draw[-,black,thin] (3,3) -- (3,-.1);
\node at (3,-.3) (213) {$H_{\alpha_1+2\alpha_2,3}$};

\draw[-,black,thin] (2,0) -- (1-.1,1+.1);
\node at (1-.2,1+.2) (111) {$H_{\alpha_1+\alpha_2,1}$};
\draw[-,black,thin] (4,0) -- (2-.1,2+.1);
\node at (2-.2,2+.2) (112) {$H_{\alpha_1+\alpha_2,2}$};

 \draw [thick, draw=black, fill=gray, fill opacity=0.2]
       (0,0) -- (1,0) -- (1,1) -- cycle;
\filldraw (0,0) circle (.05);
\node at (.33,.11) {$0$};
\node at (.11,.33) {$1$};
\node at (-.11,.33) {$3$};
\node at (-.33,.11) {$4$};
\node at (.33,-.11) {$2$};
\node at (.11,-.33) {$3$};
\node at (-.11,-.33) {$5$};
\node at (-.33,-.11) {$6$};
\node at (1.67,.11) {$1$};
\node at (3.67,.11) {$7$};
\node at (3.67,1.89) {$8$};
\node at (1.89,.33) {$3$};
\node at (2.11,.33) {$4$};
\node at (2.33,.11) {$6$};
\node at (1.89,1.66) {$4$};
\node at (2.11,1.66) {$5$};
\node at (2.33,1.89) {$7$};
\node at (2.33,2.11) {$8$};
\draw[-,black,thin] (0,-1) -- (0,1);
\node at (0,-1.2) (210) {$H_{\alpha_1+2\alpha_2,0}$};
\draw[-,black,thin] (1,-1) -- (-1,1) node[above] {$H_{\alpha_1+\alpha_2,0}$};

\end{tikzpicture}}
\end{center}
\caption{Some alcoves for the affine Weyl group of type $C_2$; the alcoves touching the origin (marked as a black dot) correspond to the elements of the finite Weyl group.  Alcoves are labeled by the sum of the lengths of the roots indexing the hyperplanes separating them from the fundamental alcove (marked as a gray triagle); under the correspondence between alcoves and elements of $\widetilde{C}_2$, this matches the statistic $\Ell_S$.}
\label{fig:c22}
\end{figure}

We give an application of~\Cref{thm:weighted} to twisted Chevalley groups in~\Cref{cor:chevalley}.  In~\Cref{sec:extensions} we discuss extensions of~\Cref{thm:weighted} beyond Weyl groups. 

\section{Numerology}
\label{sec:num}


\subsection{Complex reflection groups and degrees}

Let $V$ be a complex vector space of dimension $n$.  A \defn{complex reflection group} $G$ is a finite subgroup of $\mathrm{GL}(V)$ that is generated by reflections.  A complex reflection group $G$ is called \defn{irreducible} if $V$ is irreducible as a $G$-module; $V$ is then called the \defn{reflection representation} of $G$.  For the remainder of the paper, we fix $G$ an irreducible complex reflection group.

Let $\mathbb{C}[V^*]$ be the symmetric algebra on the dual vector space $V^*$, and write $\mathbb{C}[V^*]^G$ for its $G$-invariant subring.  By a classical theorem of Shephard-Todd~\cite{shephard1954finite} and Chevalley~\cite{chevalley1955invariants}, $G$ a subgroup of $\mathrm{GL}(V)$ is a complex reflection group if and only if $\mathbb{C}[V^*]^G$ is a polynomial ring.  For $G$ a complex reflection group, the ring $\mathbb{C}[V^*]^G$ is generated by $n$ algebraically independent polynomials---the \defn{degrees} of these polynomials are invariants of $G$, denoted $d_1\leq \cdots \leq d_n.$  

The \defn{coinvariant algebra} is the quotient $\mathbb{C}[V^*]_G:=\mathbb{C}[V^*]/\mathbb{C}[V^*]_+^G$, where $\mathbb{C}[V^*]_+^G$ is the ideal generated by all $G$-invariants with no constant term.  As an ungraded $G$-representation, the coinvariant algebra is isomorphic to the regular representation of $G$, and its Hilbert series is given by \[\sum_{i\geq 0} \mathrm{dim}\mathbb{C}[V^*]_G^i= \prod_{i=1}^n \frac{q^{d_i}-1}{q-1}.\]   In particular, we have the numerology $|G|=\prod_{i=1}^n d_i$.

\subsection{Exponents}
 The \defn{fake degrees} of an irreducible representation $U$ of $G$ are the degrees in which $U$ appears in $\mathbb{C}[V^*]_G$.  The fake degrees for the reflection representation $V$ are the \defn{exponents} of $G$, satisfying $e_i+1=d_i$.  The fake degrees for its complex conjugate $V^*$ are the \defn{co-exponents}.   Together, these satisfy the numerology \begin{equation}\sum_{i=1}^n e_i = \left|\mathcal{R}\right| \text{ and } \sum_{i=1}^n e_i^* = \left|\mathcal{H}\right|,\label{eq:exp}\end{equation} where $\mathcal{H}$ is the set of reflecting hyperplanes of $G$ and $\mathcal{R}$ is its set of reflections.

In~\cite{solomon1963invariants}, Solomon gave a uniform proof of Shephard-Todd's \Cref{thm:solomon}, showing that the (co)exponents could be read directly from the group~\cite{shephard1954finite}. 


\subsection{Reflexponents}
\label{sec:reflex}

Recall that an irreducible complex reflection group $G$ can be generated by $n$ or $n{+}1$ reflections; $G$ is said to be \defn{well-generated} if it can be generated by $n$ reflections.  Call $S \subseteq \mathcal{R}$ a \defn{minimal reflection generating set} if $S$ is a generating set of reflections of minimal size (among generating sets of reflections).

The reflecting hyperplanes $\mathcal{H}$ of the group $G$ are broken into at most three $G$-orbits $\mathcal{H}/G=\{\mathcal{H}_\epsilon\}_{\epsilon \in \{s,t,u\}}$ (where we use the indexing from~\cite{micheltable}).  Write $\mathcal{R}_\epsilon$ for the set of reflections whose reflecting hyperplane lies in a fixed orbit $\mathcal{H}_\epsilon \in \mathcal{H}/G$.   For a minimal reflection generating set $S$, set $n^S_{\epsilon}=|S\cap \mathcal{R}_\epsilon|$.

\begin{fact} For any two minimal reflection generating sets $S$ and $T$, we have $n^S_\epsilon=n^T_\epsilon$ for each hyperplane orbit $\mathcal{H}_\epsilon \in \mathcal{H}/G$.
\label{prop:minimal}
\end{fact}
\begin{proof}
There is nothing to show when there is only one hyperplane orbit.  The fact is immediate for the well-generated two-dimensional complex reflection groups $G(2b,2b,2)$, $G_5$, $G_6$, $G_9$, $G_{10}$, $G_{14}$, $G_{17}$, $G_{18}$, and $G_{21}$.  Shi and his students have classified (congruence classes of) all minimal reflection generating sets for $G(a,1,n)$, $G(ab,b,n)$, $G_7$, $G_{11}$, $G_{13}$, $G_{15}$, $G_{19}$, and $G_{26}$ (this work is surveyed in~\cite{shi2007presentations}), from which the fact follows.  We checked the fact for $G_{28}$ by computer.   
\end{proof}

By \Cref{prop:minimal}, the number of generators intersecting the hyperplane orbit $\mathcal{H}_\epsilon$ doesn't depend on the chosen minimal reflection generating set---we therefore denote this number by $n_\epsilon$.

A \defn{parabolic subgroup} of $G$ is a subgroup fixing pointwise some subset $V'$ of $V$; its \defn{dimension} is $\dim V-\dim V'$.  When $W$ is not-well-generated, it can happen that any parabolic subgroup generated by reflections from $\mathcal{R}_\epsilon$ has dimension strictly less than $n_\epsilon$ (this happens, for example, for a certain orbit in $G_{13}$).  The following definition excludes such cases, which are revisited in~\Cref{sec:extensions}.
\begin{definition}
\label{def:well_restricted}
We say that $\mathcal{H}_\epsilon$ is \defn{well-restricted} if there is some parabolic subgroup generated by a subset of $\mathcal{R}_\epsilon$ that is minimally generated by $n_\epsilon$ reflections.
\end{definition}

We have the following characterization for when $\mathcal{H}_\epsilon$ is well-restricted.
\begin{fact}
\label{prop:well-restricted}
For a finite irreducible complex reflection group $G$, $\mathcal{H}_\epsilon$ is well-restricted \emph{except} if
\begin{itemize}
\item $G=G(ab,b,n)$ with $a, b > 1$, $n>2$, and $\epsilon = t$, or
\item $G=G_{13}$ and $\epsilon = t$.
\end{itemize}
\end{fact}
\begin{proof}
Fix $S$ the the generating set given in~\cite{micheltable} (see also~\cite[Appendix 2]{broue1997complex}).

In each exceptional case except $G_{13}$, the intersection $S \cap \mathcal{R}_\epsilon$ generates the desired parabolic subgroup.  For $G_{13}$ with $\epsilon=t$ (or $u$), $n_t=2$, but $G_{13}$ has no parabolic subgroup that is minimally generated by two reflections.

Consider now $G(ab,b,n)$.  If $a=1$ then there is only one orbit of hyperplanes unless $n=2$ and $b$ is even, and we are in the case of dihedral groups---but then both $\mathcal{H}_s$ and $\mathcal{H}_t$ are well-restricted.

Otherwise $a>1$.  If $\epsilon=s$, then $\mathcal{R}_s$ consists of the $n$ diagonal matrices with $n-1$ ones on the diagonal along with a single primitive $a$th root of unity.  Any one of these generates a parabolic subgroup isomorphic to $G(a,1,1)\cong \mathbb{Z}_a$.

If $\epsilon=t$, then either $b=1$ or $b>1$.  If $b=1$, then $G(a,1,n)$ is well generated, $n_t=n-1$ and the group generated by $\{t_2,\ldots,t_n\}\subset \mathcal{R}_t$ is a parabolic subgroup (see~\cite[Theorem 4.2]{taylor2012reflection}), isomorphic to the symmetric group $G(1,1,n)\cong \mathfrak{S}_n$.  If $b>1$, $n_t=n$---but the group generated by $\{t_2,t_2',\ldots,t_n\}\subset \mathcal{R}_t$ is now \emph{not} parabolic (it is isomorphic to $G(ab,ab,n)$).  In fact, it follows from the characterization of parabolic subgroups of $G(ab,b,n)$ given in~\cite[Theorem 3.11]{taylor2012reflection} that the only parabolic subgroups of $G(ab,b,n)$ that are minimally generated by $n$ reflections are conjugate to the reflection subgroup generated by $\{s,t_2,t_2',\ldots,t_{n-1}\}$, isomorphic to $G(ab,b,n-1)$.  In particular, any parabolic subgroup of $G(ab,b,n)$ minimally generated by $n$ reflections uses a reflection from the conjugacy class $\mathcal{R}_s$.
\end{proof}
Note that all orbits are well-restricted when $G$ is well-generated.

\begin{fact}
When $\mathcal{H}_\epsilon$ is well-restricted, the parabolic subgroup $G_\epsilon$ is unique up to conjugacy.  Furthermore, there is an $n_\epsilon$-dimensional irreducible representation $V_\epsilon$ of $G$ supported on $\mathcal{R}_\epsilon$, whose restriction to $G_\epsilon$ is the reflection representation.
\end{fact}

\begin{proof}
Uniqueness up to conjugacy is immediate when $\mathcal{H}_\epsilon$ is the conjugacy class of a single reflection.  The result follows for the group $G_{28}$ because it is the complexification of a real reflection group (all parabolic subgroups are conjugate to a standard parabolic subgroup); it follows by inspection for the group $G_{26}$ and $G(a,1,n)$.

We turn now to the existence of the $n_\epsilon$-dimensional irreducible representations.  The one-dimensional representations are simple to construct.  The existence of the two-dimensional irreducible representations were confirmed for $G_{26}$ and $G_{28}$ using~\cite{MR99m:20017}.  We produce the representation for $G(a,1,n)$ with $a>1$ and the orbit $\mathcal{H}_t$ using the reflection representation on the copy of the symmetric group $G(1,1,n)$ inside $G(a,1,n)$ generated by $\{t_2,\ldots,t_n\}$.
\end{proof}

As in the introduction, we define the \defn{reflexponents} $\{\epsilon_i\}_{i=1}^{n_\epsilon}$ to be the ${n_\epsilon}$ fake degrees of $V_\epsilon$, and the \defn{co-reflexponents} $\{\epsilon_i^*\}_{i=1}^{n_\epsilon}$ to be the fake degrees of $V_\epsilon^*$.  These are listed in~\Cref{fig:table} (the fake degrees for the infinite family were determined in~\cite{malle1995unipotente}).  An analogue of~\Cref{eq:exp} holds for (co)reflexponents.

\begin{figure}[htbp]
\[\scalebox{1}{$\begin{array}{cccccc}
G &  e_i/ e_{n-i}^* & s_i/s_{n_s-i}^* & t_i/t^*_{n_t-i} & u_i/u^*_{n_u-i} \\ \hline
\begin{array}{c}G(a,1,n)\\ \scalebox{0.8}{$a{>}1$}\end{array} &  \scalebox{0.8}{$\begin{array}{c} a{-}1,\scalebox{0.5}{$\ldots$},na{-}1 \\ (n{-}1)a{+}1 ,\scalebox{0.5}{$\ldots$},1\end{array}$} &   \begin{array}{c} (a{-}1)n \\ n \end{array} & \scalebox{0.8}{$\begin{array}{c} a,\scalebox{0.5}{$\ldots$},(n{-}1)a \\ (n{-}1)a,\scalebox{0.5}{$\ldots$},a \end{array}$} & \\ \hline
\begin{array}{c}G(m,b,n)\\ \scalebox{0.8}{$\text{write } m=ab$}\\\scalebox{0.8}{$d=\gcd(a,b)$}\\ \scalebox{0.8}{$a,b{>}1,n{>}2$} \end{array}  & \scalebox{0.8}{$\begin{array}{c} m{-}1,\scalebox{0.5}{$\ldots$},(n{-}1)m{-}1,na{-}1 \\ 1,m{+}1,\scalebox{0.5}{$\ldots$},(n{-}1)m{+}1 \end{array}$}  & \begin{array}{c} (a{-}d)n \\ n \end{array} & * & \\ \hline
G(2b,2b,2) & \begin{array}{c} 1,2b{-}1 \\2b{-}1,1\end{array} & \begin{array}{c}b\\ b\end{array} &\begin{array}{c}b\\ b\end{array} &\\ \hline
G_5 
& \begin{array}{c}5,11\\ 7,1\end{array} & \begin{array}{c}8\\4\end{array} & \begin{array}{c}8\\4\end{array} & \\ \hline 
G_6 & \begin{array}{c}3,11\\9,1\end{array} & \begin{array}{c}6\\6 \end{array} & \begin{array}{c}8\\4\end{array} & \\ \hline 
G_7  & \begin{array}{c}11,11\\13,1\end{array} & \begin{array}{c}6\\6 \end{array} & \begin{array}{c}8\\4\end{array}  & \begin{array}{c}8\\4\end{array}\\ \hline 
G_9  & \begin{array}{c}7,23\\17,1\end{array} & \begin{array}{c}12\\12\end{array} & \begin{array}{c}18\\ 6\end{array}  \\ \hline 
G_{10}  & \begin{array}{c}11,23\\13,1\end{array} & \begin{array}{c}16\\8\end{array} & \begin{array}{c}18\\6\end{array} \\ \hline 
G_{11}  & \begin{array}{c}23,23\\25,1\end{array} & \begin{array}{c}12\\12\end{array} &  \begin{array}{c}16\\8\end{array} & \begin{array}{c}18\\6\end{array} \\ \hline 
G_{13} & \begin{array}{c} 7,11 \\ 17,1 \end{array} & \begin{array}{c}6 \\ 6 \end{array} & *\\ \hline
G_{14}  & \begin{array}{c}5,23\\19,1\end{array} & \begin{array}{c}12\\12\end{array} & \begin{array}{c}16\\8\end{array} & \\ \hline 
G_{15}  & \begin{array}{c}11,23\\25,1\end{array} & \begin{array}{c}12\\12\end{array} & \begin{array}{c}16\\8\end{array} & \begin{array}{c}6\\6\end{array}\\ \hline 
G_{17} & \begin{array}{c}19,59\\41,1\end{array} & \begin{array}{c}30\\30\end{array} & \begin{array}{c}48\\12\end{array} \\ \hline 
G_{18} & \begin{array}{c}29,59\\31,1\end{array} & \begin{array}{c}40\\20\end{array} & \begin{array}{c}48\\12\end{array} \\\hline 
G_{19} & \begin{array}{c}59,59\\61,1\end{array} & \begin{array}{c}30\\30\end{array} & \begin{array}{c}40\\20\end{array} & \begin{array}{c}48\\12\end{array} \\\hline 
G_{21} & \begin{array}{c}11,59\\49,1\end{array}& \begin{array}{c}30\\30\end{array} & \begin{array}{c}40\\20\end{array} \\\hline 
G_{26} & \begin{array}{c}5,11,17\\13,7,1\end{array} & \begin{array}{c}9\\9\end{array}  & \begin{array}{c}9,15\\9,3\end{array}\\ \hline
G_{28} & 1,5,7,11 & \begin{array}{c}4,8 \\ 8,4 \end{array} & \begin{array}{c}4,8 \\ 8,4 \end{array}




\end{array}$}\]
\caption{The (co)exponents and (co)reflexponents for the complex reflection groups whose reflecting hyperplanes break into more than one $G$-orbit, using the conventions of~\cite{micheltable}.  The entries marked with a star correspond to orbits that are not well-restricted (see~\Cref{sec:extensions} and \Cref{fig:non_restricted}).}
\label{fig:table}
\end{figure}

\begin{fact}
Let $\mathcal{H}_\epsilon$ be well-restricted and $\mathcal{R}_\epsilon$ the corresponding orbit of reflections.  Then \[\sum_{i=1}^n \epsilon_i^* = |\mathcal{H}_\epsilon| \text{ and } \sum_{i=1}^n \epsilon_i = |\mathcal{R}_\epsilon|.\]
\label{fact:hr_well_restricted}
\end{fact}

\begin{proof}
Case-by-case check, using~\Cref{fig:table}.
\end{proof}




\subsection{Exponents and reflexponents in well-generated groups} 
It is known that $G$ is well-generated if and only if \begin{equation}e_i+e_{n+1-i}^*=e_n+1.\label{eq:well-gen}\end{equation}  An analogue of~\Cref{eq:well-gen} holds for (co)reflexponents.

\begin{fact}
\label{prop:well-gen}
Let $G$ be well-generated and $\mathcal{H}_\epsilon$ well-restricted.  Then \[\epsilon_i+\epsilon_{n_\epsilon+1-i}^*=e_n+1.\]
\end{fact}

\begin{proof}
Case-by-case check using~\Cref{fig:table}.
\end{proof}

\subsection{Exponents in Weyl groups}



Let $\Phi$ be an $n$-dimensional irreducible crystallographic root system with \defn{simple roots} $\Delta$ and \defn{positive roots} $\Phi^+$.   The positive roots are ordered by $\alpha \leq \beta$ if $\beta-\alpha$ is a nonnegative sum of simple roots.  There is a unique \defn{highest root} $\widetilde{\alpha} \in \Phi^+$, defined by the property that $\beta\leq \widetilde{\alpha}$ for any $\alpha \in \Phi^+$.  If $\alpha=\sum_{i=1}^n a_i \alpha_i$, we define its \defn{height} $\hgt(\alpha)$ to be the sum $\sum_{i=1}^n a_i$.  Then $h=\hgt(\widetilde{\alpha})+1.$

Shapiro and Steinberg independently observed that the exponents of $W$ could be computed by taking the partition dual to the heights of the positive roots~\cite[Section 9]{steinberg1959finite}.  This duality was later uniformly explained by Kostant in~\cite{kostant2009principal}.  Analogously, we define the \defn{short exponents} $\epsilon_1 \leq \epsilon_2 \leq \cdots \leq \epsilon_{m}$ to be the partition dual to the heights of the short roots in the \emph{dual} root system.  The short exponents are easily computed, and are listed in~\Cref{fig:table_weyl}.  For convenience, we define \defn{short degrees} $\delta_i:=\epsilon_i{+}1$.

Note that when $G$ is a Weyl group, all hyperplanes corresponding to roots of the same length occur in the same $G$-orbit, and so $\mathcal{H}/G$ recovers the partition into long and short roots.  In fact, the reflexponents (defined using fake degrees) agree with the short exponents (defined using heights of roots).

\begin{proposition}
\label{thm:small_agree}
Let $\Phi$ be an irreducible crystallographic root system with Weyl group $W$, and let $\mathcal{H}_\epsilon$ be the orbit of hyperplanes corresponding to the long roots of $\Phi$.  Then the set of reflexponents for $W_\epsilon$ is equal to the set of short exponents of $\Phi$.
\end{proposition}

\begin{proof}
The result follows from comparison of~\Cref{fig:table} with~\Cref{fig:table_weyl}.
\end{proof}

\begin{figure}[htbp]
\[\begin{array}{ccccccc}
W & \text{Diagram} & r & e_i & \epsilon_i \\ \hline
W(B_n) = G(2,1,n) & \reflectbox{\dynkin{B}{**.**}} & 2 & 1,3,\dots,2n{-}1 & 2,4,\ldots,2n{-}2 \\ \hline
W(C_n) = G(2,1,n) & \reflectbox{\dynkin{C}{**.**}} & 2 & 1,3, \ldots,2n{-}1 & n\\ \hline
W(F_4) = G_{28} & \dynkin{F}{4} & 2 & 1,5,7,11 & 4,8 \\ \hline
W(G_2) = G(6,6,2) & \dynkin{G}{2} & 3 & 1,5 & 3
\end{array}\]
\caption{The exponents and short exponents for the Weyl groups whose roots have two different lengths.  Here, the short exponents $\epsilon_i$ are defined as the dual partition of the heights of the short roots in the dual root system.}
\label{fig:table_weyl}
\end{figure}

\begin{remark}
When $G$ has a presentation by reflections $r_1,\ldots,r_n$ with an automorphism $\sigma$ of the presentation with the property that $\sigma^j(r_i)$ commutes with $r_i$, we can ``fold'' this presentation of $G$ to produce a new complex reflection group $G_\sigma$.  Then it appears that the set of (co)exponents of $G$ is related to the union of the (co)exponents and (co)reflexponents of $G_\sigma$---this is easily explained when $G$ is a Weyl group and we may use the fact that exponents are dual to the heights of roots, but we have no explanation when $G$ is not the complexification of a real reflection group.  (See also the discussion around~\cite[Proposition 6.1]{orlik1980unitary}; there may be some connection to Springer's theory of regular elements~\cite[Theorem 4.2 (iii)]{springer1974regular}.)

For example, the exponents of the Weyl groups $A_{2n-1}$, $D_{n+1}$, $D_4$, $E_6$ matches the union of the exponents and short exponents of $B_n$, $C_n$, $G_2$, and $F_4$, respectively.  But the exponents of the complex reflection group $G_{25}$ also matches the union of the exponents and reflexponents of $G_5$, while the co-exponents of $G_{25}$ are the union of the co-exponents and co-reflexponents of $G_5$.

\end{remark}

\section{Proof of \Cref{thm:solomon2}}

Given an orbit of hyperplanes $\mathcal{H}_\epsilon$ with $G_\epsilon$ well-restricted, define the generating functions
\begin{align*}
\mathcal{M}_\epsilon(G;x,y)&:= \sum_{g \in G} \left(x/y\right)^{\mathcal{M}_{V_\epsilon}(g)}y^{\mathcal{M}_V(g)},\\
\mathcal{M}^*_\epsilon(G;x,y)&:= \sum_{g \in G} \mathrm{det}(g)\left(x/y\right)^{\mathcal{M}_{V_\epsilon}(g)}y^{\mathcal{M}_V(g)}.
\end{align*}

{
\renewcommand{\thetheorem}{\ref{thm:solomon2}}
\begin{theorem}
For $G$ a finite irreducible complex reflection group and $\mathcal{H}_\epsilon$ a well-restricted orbit of reflecting hyperplanes, there is a reindexing of the (co)exponents by $1,\ldots,n$ (with undefined (co)exponents taken to be zero), such that
\begin{align*}
 \mathcal{M}_\epsilon(G;x,y) = \prod_{i=1}^n \Big(1+\epsilon_i x+(e_i{-}\epsilon_i)y\Big) \text{, }
  \mathcal{M}^*_\epsilon(G;x,y)= \prod_{i=1}^n \Big(1-\epsilon^*_i x-(e_i^*{-}\epsilon_i^*)y\Big).
 \end{align*}
\end{theorem}
\addtocounter{theorem}{-1}
}

\begin{remark}
Except for the case of $G(ab,b,n)$ for $b>1$, if we have $\epsilon_1 \leq \cdots \leq \epsilon_{n_\epsilon}$ and $\epsilon_1 \leq \cdots \leq \epsilon_{n_\epsilon}$, then the reindexing of the (co)reflexponents in~\Cref{thm:solomon2} is to add $n-n_\epsilon$ to the index $i$ of $\epsilon_i$.   In the case of $G(ab,b,n)$ and the well-restricted orbit $\mathcal{H}_s$, we must instead associate the single reflexponent to the exponent $an-1$.
\end{remark}

\begin{proof}
We verify the result case-by-case.

{\bf Groups with a single hyperplane orbit.}  The result follows from~\Cref{thm:solomon} in the case when there is a single $G$-orbit of reflecting hyperplanes.  We are therefore reduced to the cases listed in~\Cref{fig:table}.

{\bf Exceptional groups.}
We used a computer to verify the result for the groups $G(5)$, $G(6)$, $G(7)$, $G(9)$, $G(10)$, $G(11)$, $G(14)$, $G(15)$, $G(17)$, $G(18)$, $G(19)$, $G(21)$, and $G(26)$~\cite{sagemath,Sch97,MR99m:20017}.

{\bf Dihedral groups.}
The dihedral group $G(2b,2b,2)$ has $2b$ reflecting hyperplanes divided into two orbits $\mathcal{H}_s$ and $\mathcal{H}_t$, each of size $b$---but both orbits give the same (co)reflexponents.  We therefore consider only the orbit corresponding to the reflection $s$; the representation $V_s$ is defined by $s \mapsto -1$ and $t \mapsto 1$.

The identity contributes $1$ to $\mathcal{M}_s(G(2b,2b,2);x,y)$ and $\mathcal{M}^*_s(G(2b,2b,2);x,y)$, while the reflections of $G(2b,2b,2)$ together contribute $\pm b(x+y)$.  There are $2b-1$ elements remaining, each of reflection length $2$.  An element $w$ contributes $xy$ to the sum if and only if $w$ has a reduced word with an odd number of copies of $s$.  Such elements are of the form $(st)^{2i-1}$ and $(ts)^{2i-1}$ for $1\leq i\leq \frac{b+1}{2}$.  If $b$ is odd, then the long element $w_\circ=(ts)^b=(st)^b$ is double counted, so that in either case we have exactly $b$ such elements.  The remaining $(b-1)$ elements now each contribute $y^2$, giving the desired formulas:
\begin{align*}
\mathcal{M}_s(G(2b,2b,2);x,y) &= 1+b(x{+}y)+b x y+(b{-}1) y^2 = \big(y{+}1\big) \big(bx + (b {-} 1) y+1\big), \\
\mathcal{M}_s^*(G(2b,2b,2);x,y) &= 1-b(x{+}y)+b x y+(b{-}1) y^2 = \big(y-1\big) \big(bx + (b {-} 1) y-1\big).
\end{align*}

Even though they have only one orbit of reflecting hyperplanes, we can perform a similar computation for the odd dihedral groups $G(2b+1,2b+1,2)$ by expressing elements as reduced words in simple reflections and substituting $s \mapsto -1$ and $t \mapsto 1$.  Except for the long element, reduced words are unique---as long as we choose the reduced word for the long element using the fewest number of copies of the simple reflection $s$, we obtain similar factorizations into two linear factors.

\allowdisplaybreaks
{\bf The infinite family $G(ab,b,n)$, $a>1$.}
Write $m=ab$, $\gc=\gcd(a,b)$, and let $\zeta$ be a primitive $m$th root of unity.  The group $G(m,b,n)$ has two orbits of reflecting hyperplanes: \[\mathcal{H}_s=\big\{x_i=0 : 1 \leq i \leq n\big\} \text{ and } \mathcal{H}_t=\big\{x_i=\zeta^k x_j : 1 \leq i<j \leq n \text{ and } 0 \leq k < r\big\}.\]  The reflection representation of $G(m,b,n)$ coincides with the group of matrices with entries in $\{0,\zeta,\zeta^2,\ldots,\zeta^{m}\}$ with exactly one non-zero entry in each row and column, such the the sum of the exponents of the nonzero entries is zero modulo $b$.  These may be represented as permutations decorated by integers modulo $m$, and so inherit standard notions regarding permutations, such as cycle decompositions.  We prove the statement for the reflexponents for both orbits by refining the original combinatorial argument due to Shephard-Todd~\cite[Section 9]{shephard1954finite}.  Recall that an element $g \in G(m,b,n)$ has $\mathcal{M}_V(g)=r$ iff it has exactly $r$ cycles such that the sum of the decorating integers from those cycles is zero modulo $m$.


  %

{\bf $G(a,1,n)$ and the orbit $\mathcal{H}_t$.}
Consider a permutation $\sigma$ with $n-r+j$ cycles, and designate $n-r$ cycles to have decoration sum zero modulo $a$.  For each cycle, we can decorate all but one element freely.  The remaining decorating integer is forced for the cycles with decoration sum zero modulo $a$; the remaining decorating integer for the other $j$ cycles must be chosen so that the decoration sum is nonzero modulo $a$.  The number of elements of $G(a,1,n)$ with $\mathcal{M}(g)=r$ is therefore \[\sum_{s=0}^r \binom{n-r+j}{j} a^{r-j} (a-1)^j \S_{r-j}(n),\] where $\S_{r-j}(n)$ is the Stirling number counting the number of elements of $\mathfrak{S}_n$ with $n-r+j$ cycles.  For the orbit $\mathcal{H}_t$, we may determine the $(n{-}1)$-dimensional $\epsilon$-reflection representation using the underlying permutation matrix (obtained by replacing all roots of unity in the reflection representation by 1).  Following~\cite{shephard1954finite}, we compute
\begin{align*}
\mathcal{M}_t(G(a,1,n);x,y) &= \sum_{r=0}^n y^r\sum_{j=0}^r \binom{n-r+j}{j} a^{r-j} (a-1)^s \S_{r-j}(n) \left(\frac{x}{y}\right)^{r-j}
\\
&=\sum_{i=0}^n \S_i(n)(ax)^i \left(\sum_{j=0}^{n-i} \binom{n-i}{j} (a-1)^j y^j \right)\\
&= \sum_{i=0}^n \S_i(n) (ax)^i \big(1+(a-1)y\big)^{n-i} \\
&= (1+(a-1)y)^n \sum_{i=0}^n \S_i(n) \left(\frac{ax}{1+(a-1)y}\right)^i\\
&= (1+(a-1)y)^n \prod_{i=1}^{n-1} \left(1+\frac{iax}{1+(a-1)y}\right)\\
&= \Big(1+(a-1)y\Big) \prod_{i=1}^{n-1} \Big(1+(ia)x+(a-1)y\Big),
\end{align*}
where we have used the fact that $\sum_{i=0}^n \S_i(n) q^i = \prod_{i=1}^{n-1} (1+iq).$

On the other hand, since the decoration sum is zero modulo $a$ for the chosen $n-r$ cycles, the determinant of an element is given by the product of the nonzero decoration sums times the sign of the underlying permutation matrix.  To compute the product of the nonzero decoration sums, we find the sum over all multisubsets of $\{p_1,\ldots,p_j\} \subseteq \{1,2,\ldots,a-1\}$ to be \[\sum_{\{p_i\} \in \left(\!\!\binom{\{1,2,\ldots,m-1\}}{j}\!\!\right)} \prod_{i=1}^j q^{p_i}=q^j (1+q+\cdots+q^{a-2})^j,\] which yields $(-1)^j$ upon the substitution $q=\zeta$.  We now compute
\begin{align*}
\mathcal{M}^*_t(G(a,1,n);x,y) &= \sum_{i=0}^n \S_i(n)(-ax)^i \left(\sum_{j=0}^{n-i} \binom{n-i}{j}(-1)^j y^j \right)\\
&= \sum_{i=0}^n \S_i(n) (-ax)^i \big(1-y\big)^{n-i} \\
&= (1-y)^n \sum_{i=0}^n \S_i(n) \left(\frac{-ax}{1-y}\right)^i\\
&= (1-y)^n \prod_{i=1}^{n-1} \left(1-\frac{iax}{1-y}\right)\\
&= \Big(1-y\Big) \prod_{i=1}^{n-1} \Big(1-(ia)x-y\Big).
\end{align*}

{\bf $G(ab,b,n)$ and the orbit $\mathcal{H}_s$.}
For the orbit $\mathcal{H}_s$, we can read the one-dimensional $\epsilon$-reflection representation from the reflection representation of $g \in G(m,b,n)$ by considering the total sum of the exponents on the powers of zeta in the matrix for $g$---if the sum is divisible by $a$, then $g$ contributes $y^{\mathcal{M}_V(g)}$; otherwise, it contributes $xy^{\mathcal{M}_V(g)-1}$.  Again fixing a permutation $\sigma$ with $n-r+j$ cycles of which $n-r$ cycles have decoration sum zero modulo $m$, we need to compute the number of ways to decorate the remaining $j$ cycles.  The number of ways to choose $s$ elements with repetition from $\{1,\ldots,m-1\}$ so that the sum is both zero modulo $a$ and zero modulo $b$ is $m_j:=
\gc \frac{(m - 1)^j - (-1)^j}{m} + (-1)^j$, leaving $n_j:=(a-\gc)\frac{(m-1)^j-(-1)^j}{m}$ ways to choose $s$ elements.  We may now compute $\mathcal{M}_s(G(ab,b,n);x,y)$ as

\begin{align*}
&\sum_{r=0}^n y^r \Bigg[\sum_{j=0}^r \S_{r-j}(n) \binom{n-r+j}{j} m^{r-j} \left(m_j + \frac{x}{y} n_j\right) \Bigg]\\
=& \sum_{i=0}^n \S_{i}(n)(my)^{i} \Bigg[\sum_{j=0}^{n-i} y^{j} \binom{n-i}{j} \left(m_j + \frac{x}{y} n_j\right) \Bigg]\\
=&-\Bigg[\frac{(a-\gc)x(1-y)^n}{my}+ \frac{\gc(1-y)^n}{m}\Bigg]\sum_{i=0}^n \S_{i}(n)\left(\frac{my}{1-y}\right)^{i}+\\
&+\Bigg[\frac{(a{-}\gc)x(1{+}(m{-}1)y)^n}{my}+\frac{\gc(1{+}(m{-}1)y)^n}{m}\Bigg]\sum_{i=0}^n \S_{i}(n)\left(\frac{my}{1+(m-1)y}\right)^{i}+\\
&+(1-y)^n\sum_{i=0}^n \S_{i}(n)\left(\frac{my(1+(m-1)y}{(1-y)(1+(m-1)y)}\right)^{i}\\
=&\Big(1 + (a-d) n x + (d n-1) y\Big)\prod_{i=1}^{n-1} (1+(im-1)y).
\end{align*}

The identity for the co-reflexponent follows similarly.\qedhere

\end{proof}







\section{Proof of \Cref{thm:weighted}}
For $\Phi$ an irreducible crystallographic root system with Weyl group $W$ and affine Weyl group $\widetilde{W}$, define the generating functions
\begin{align*}
\Ell_S(W;q):=\sum_{w \in W} q^{\Ell_S(w)} \text{ and } \Ell_S(\widetilde{W};q):=\sum_{w \in \widetilde{W}} q^{\Ell_S(w)}.
\end{align*}

{
\renewcommand{\thetheorem}{\ref{thm:weighted}}
\begin{theorem}
For an irreducible crystallographic root system $\Phi$,
\begin{align*}\Ell_S(W;q) &= \prod_{i=1}^{n_\epsilon}\left(\frac{1+q^{\delta_i}+\cdots+q^{\delta_i(r-1)}}{1+q+\cdots+q^{r-1}}\right)\prod_{i=1}^n \left(\frac{q^{d_i}-1}{q-1}\right),
\text{ and} \\
\Ell_S(\widetilde{W};q) &= (1-q)^{-n}\prod_{i=1}^{n_\epsilon} \left(\frac{q^{\epsilon_i}-1}{q^{\epsilon_i r}-1}\right) \prod_{i=1}^n \left(\frac{q^{d_i}-1}{q^{e_i}-1}\right).
\end{align*} 
\end{theorem}
\addtocounter{theorem}{-1}
}

\begin{proof}
There are several ways to proceed directly case-by-case---for example, we could use the uniqueness of factorization of an element $w \in W$ into a product of an element from a parabolic subgroup and an element for a parabolic quotient (see, for example, \cite[Chapter 7]{bjorner2006combinatorics})---but the easiest way to prove these formulas is by specializing some beautiful results due to Macdonald.

If $\alpha=\sum_{i=1}^n a_i \alpha_i$ expresses a positive root $\alpha$ as a sum of simple roots, write $\Hgt(\alpha)=\sum_{i=1}^n a_i \|\alpha_i\|^2$, where we recall that short roots are normalized to have length $1$.  A specialization of \cite[Theorem 2.4]{macdonald1972poincare} (which allows more freedom in weighting the positive roots) now shows that 
\[\Ell_S(W;q) = \prod_{\alpha \in \Phi^+} \frac{q^{\|\alpha\|^2+\Hgt(\alpha)}-1}{q^{\Hgt(\alpha)}-1},\] and the result is easily confirmed case-by-case (one can also use the explicit formulas given in~\cite[Section 2.2]{macdonald1972poincare}, substituting in the appropriate root lengths).   The results for $\widetilde{W}$ follow similarly using~\cite[Theorem 3.3]{macdonald1972poincare}.\qedhere

\Cref{fig:computation} illustrates the calculation of~\Cref{thm:weighted} for types $B_5$ and $C_5$.

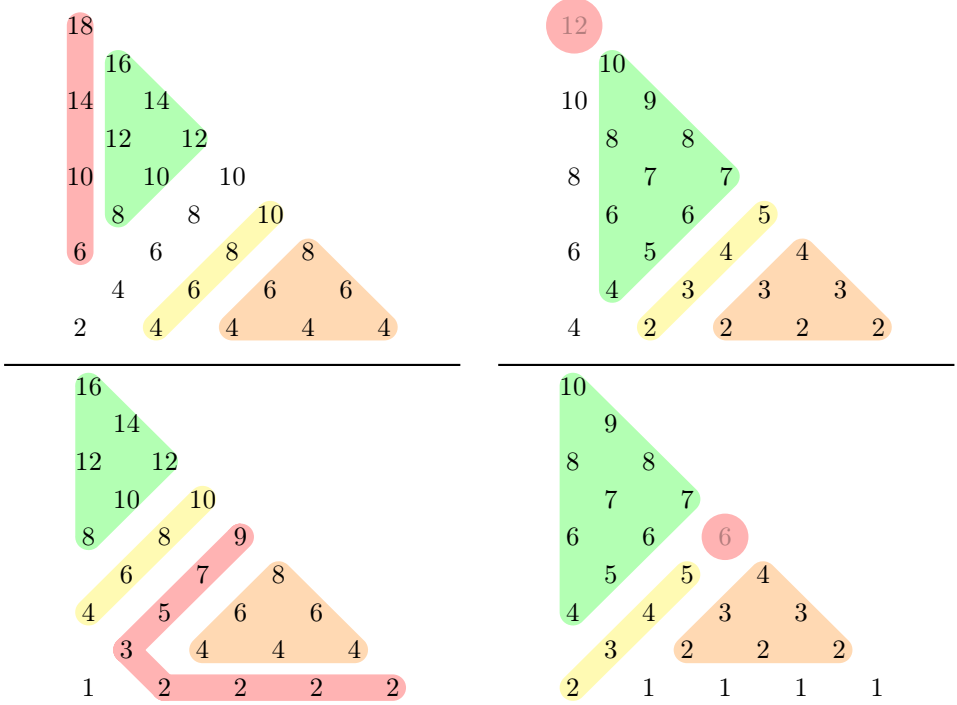
\begin{figure}[htbp]
\begin{center}
\begin{tikzpicture}[scale=.5]
\node[black] at (-0,0) (a) {4};
\node[black] at (-2,0) (b) {4};
\node[black] at (-4,0) (c) {4};
\node[black] at (-6,0) (d) {4};
\node[black] at (-8,0) (e) {2};
\node[black] at (-1,1) (f) {6};
\node[black] at (-3,1) (g) {6};
\node[black] at (-5,1) (h) {6};
\node[black] at (-7,1) (i) {4};
\node[black] at (-2,2) (j) {8};
\node[black] at (-4,2) (k) {8};
\node[black] at (-6,2) (l) {6};
\node[black] at (-8,2) (m) {6};
\node[black] at (-3,3) (n) {10};
\node[black] at (-5,3) (o) {8};
\node[black] at (-7,3) (p) {8};
\node[black] at (-4,4) (q) {10};
\node[black] at (-6,4) (r) {10};
\node[black] at (-8,4) (s) {10};
\node[black] at (-5,5) (t) {12};
\node[black] at (-7,5) (u) {12};
\node[black] at (-6,6) (v) {14};
\node[black] at (-8,6) (w) {14};
\node[black] at (-7,7) (x) {16};
\node[black] at (-8,8) (y) {18};
\draw[thick] (2,-1)--(-10,-1);
\begin{pgfonlayer}{background}
\fill[orange,opacity=0.3] \convexpath{a,c,j}{10pt};
\fill[green,opacity=0.3] \convexpath{p,x,t}{10pt};
\fill[yellow,opacity=0.3] \convexpath{d,n}{10pt};
\fill[red,opacity=0.3] \convexpath{m,y}{10pt};
\end{pgfonlayer}
\end{tikzpicture}\hspace{1em} 
\begin{tikzpicture}[scale=.5]
\node[black] at (-0,0) (a) {2};
\node[black] at (-2,0) (b) {2};
\node[black] at (-4,0) (c) {2};
\node[black] at (-6,0) (d) {2};
\node[black] at (-8,0) (e) {4};
\node[black] at (-1,1) (f) {3};
\node[black] at (-3,1) (g) {3};
\node[black] at (-5,1) (h) {3};
\node[black] at (-7,1) (i) {4};
\node[black] at (-2,2) (j) {4};
\node[black] at (-4,2) (k) {4};
\node[black] at (-6,2) (l) {5};
\node[black] at (-8,2) (m) {6};
\node[black] at (-3,3) (n) {5};
\node[black] at (-5,3) (o) {6};
\node[black] at (-7,3) (p) {6};
\node[black] at (-4,4) (q) {7};
\node[black] at (-6,4) (r) {7};
\node[black] at (-8,4) (s) {8};
\node[black] at (-5,5) (t) {8};
\node[black] at (-7,5) (u) {8};
\node[black] at (-6,6) (v) {9};
\node[black] at (-8,6) (w) {10};
\node[black] at (-7,7) (x) {10};
\node[black,circle,fill=red,opacity=0.3,text opacity=1] at (-8,8) (y) {12};
\draw[thick] (2,-1)--(-10,-1);
\begin{pgfonlayer}{background}
\fill[orange,opacity=0.3] \convexpath{a,c,j}{10pt};
\fill[green,opacity=0.3] \convexpath{i,x,q}{10pt};
\fill[yellow,opacity=0.3] \convexpath{d,n}{10pt};
\end{pgfonlayer}
\end{tikzpicture}

\begin{tikzpicture}[scale=.5]
\node[black] at (-0,0) (a) {2};
\node[black] at (-2,0) (b) {2};
\node[black] at (-4,0) (c) {2};
\node[black] at (-6,0) (d) {2};
\node[black] at (-8,0) (e) {1};
\node[black] at (-1,1) (f) {4};
\node[black] at (-3,1) (g) {4};
\node[black] at (-5,1) (h) {4};
\node[black] at (-7,1) (i) {3};
\node[black] at (-2,2) (j) {6};
\node[black] at (-4,2) (k) {6};
\node[black] at (-6,2) (l) {5};
\node[black] at (-8,2) (m) {4};
\node[black] at (-3,3) (n) {8};
\node[black] at (-5,3) (o) {7};
\node[black] at (-7,3) (p) {6};
\node[black] at (-4,4) (q) {9};
\node[black] at (-6,4) (r) {8};
\node[black] at (-8,4) (s) {8};
\node[black] at (-5,5) (t) {10};
\node[black] at (-7,5) (u) {10};
\node[black] at (-6,6) (v) {12};
\node[black] at (-8,6) (w) {12};
\node[black] at (-7,7) (x) {14};
\node[black] at (-8,8) (y) {16};
\begin{pgfonlayer}{background}
\fill[orange,opacity=0.3] \convexpath{f,h,n}{10pt};
\fill[green,opacity=0.3] \convexpath{s,w,y,v}{10pt};
\fill[yellow,opacity=0.3] \convexpath{m,t}{10pt};
\fill[red,opacity=0.3] \convexpath{q,i,d,a,d,i,q}{10pt};
\end{pgfonlayer}
\end{tikzpicture}\hspace{5em} 
\begin{tikzpicture}[scale=.5]
\node[black] at (-0,0) (a) {1};
\node[black] at (-2,0) (b) {1};
\node[black] at (-4,0) (c) {1};
\node[black] at (-6,0) (d) {1};
\node[black] at (-8,0) (e) {2};
\node[black] at (-1,1) (f) {2};
\node[black] at (-3,1) (g) {2};
\node[black] at (-5,1) (h) {2};
\node[black] at (-7,1) (i) {3};
\node[black] at (-2,2) (j) {3};
\node[black] at (-4,2) (k) {3};
\node[black] at (-6,2) (l) {4};
\node[black] at (-8,2) (m) {4};
\node[black] at (-3,3) (n) {4};
\node[black] at (-5,3) (o) {5};
\node[black] at (-7,3) (p) {5};
\node[black,circle,fill=red,opacity=0.3,text opacity=1] at (-4,4) (q) {6};
\node[black] at (-6,4) (r) {6};
\node[black] at (-8,4) (s) {6};
\node[black] at (-5,5) (t) {7};
\node[black] at (-7,5) (u) {7};
\node[black] at (-6,6) (v) {8};
\node[black] at (-8,6) (w) {8};
\node[black] at (-7,7) (x) {9};
\node[black] at (-8,8) (y) {10};
\begin{pgfonlayer}{background}
\fill[orange,opacity=0.3] \convexpath{f,h,n}{10pt};
\fill[green,opacity=0.3] \convexpath{m,y,t}{10pt};
\fill[yellow,opacity=0.3] \convexpath{e,o}{10pt};
\end{pgfonlayer}
\end{tikzpicture}
\end{center}
\caption{The example calculation for the proof of~\Cref{thm:weighted}, using \cite[Theorem 2.4]{macdonald1972poincare}.  The two diagrams on the left represent the calculation for $B_5$, while the two diagrams on the right give the calculation for $C_5$.  In both cases, the bottom diagrams contain the numbers $\Hgt(\alpha)$ for $\alpha \in \Phi^+$; the top diagrams are the numbers $|\alpha|+\Hgt(\alpha)$.  Cancellations are denoted using color.}
\label{fig:computation}
\end{figure}

\end{proof}

\begin{corollary}
Let $G_\sigma$ a twisted Chevalley group of type $Y$ over $\mathbb{F}_q$, where $G$ is simply-laced of type $X$.  Then
\[\big|G_\sigma\big| = q^{N} \prod_{i=1}^{n_\epsilon} \left(1+q^{\delta_i}+\cdots+ q^{\delta_i (r-1)}\right) \prod_{i=1}^n \left(q^{d_i}-1\right),\]
where $N$ is the number of postive roots in a root system of type $X$, while $\{d_i\}_{i=1}^n$, $\{\delta_i\}_{i=1}^{n_\epsilon}$, and $r$ are the degrees, short degrees, and the ratio of a long to a short root in a root system of type $Y$.
\label{cor:chevalley}
\end{corollary}

\begin{proof}
This follows from~\Cref{thm:weighted} by comparison with \cite[Theorem 35]{steinberg1967lectures} and \cite[Section 2.3]{macdonald1972poincare}.
\end{proof}

\section{Future Work and Generalizations}
\label{sec:extensions}
We have certainly not found the correct proofs for~\Cref{thm:solomon2}, or for~\Cref{thm:weighted}.  As Shephard writes in~\cite{shephard1956some}:

\begin{displayquote}
Sometimes a proof in general terms is known, but in the majority of cases it has been necessary to verify the properties one by one for all the irreducible groups over $\mathbb{C}$\ldots These two distinct methods will be referred to as \emph{proving} and \emph{verifying} respectively.
\end{displayquote}

In this language, we have only \emph{verified} our theorems.  But there is also evidence that we have not even found the correct formulation of~\Cref{thm:solomon2} (and, in particular, the notion of well-restrictedness), as we now explain.
\subsection{Extension of \Cref{thm:solomon2} to $G_{13}$  and $\mathcal{H}_t$}
The group $G_{13}$ is not well-restricted with respect to the orbit $\mathcal{H}_{t}=\mathcal{H}_{u}$.  Nevertheless, there is still a two-dimensional (real) irreducible representation $U_t$ of $G$ that is supported on $\mathcal{R}_{t}$ and restricts to the reflection representation of a (nonparabolic) reflection subgroup of $G_{13}$:
\[s \mapsto \left(\begin{array}{ccc}1&0\\0&1\end{array}\right), t \mapsto \left(\begin{array}{ccc}-1&1\\0&1\end{array}\right), u \mapsto \left(\begin{array}{ccc}1&0\\1&-1\end{array}\right)\]

The generating function recording the reflection representation $V$ of $G_{13}$ and the representation $U_t$ still factors into linear terms using the ``reflexponents'' given in~\Cref{fig:non_restricted}, and these terms still encode the fake degrees of $U_t$ and $V$.  

\begin{align*}\sum_{g \in G_{13}} \left(x/y\right)^{\mathcal{M}_{U_t}(g)}y^{\mathcal{M}_V(g)}&=(1+8x + 3y)(1+4x + 3y)\\
 \sum_{g \in G_{13}} \mathrm{det}(g)\left(x/y\right)^{\mathcal{M}_{U_t}(g)}y^{\mathcal{M}_V(g)} &= (1-12x-5y)(1-y).
\end{align*}

Remarkably, the generating function weighted by $\mathrm{det}(g)$ also factors into linear terms, and recovers the fake degrees of the irreducible one-dimensional representation $U'_t$ defined by \[s \mapsto 1, t\mapsto -1, u\mapsto -1.\]

\subsection{Extension of \Cref{thm:solomon2} to $G(ab,b,n)$ and $\mathcal{H}_t$}
The group $G=G(ab,b,n)$ for $n>2$ and $a,b>1$ is not well-restricted with respect to the orbit $\mathcal{H}_t$.  Still, $G$ does have an irreducible representation $U_t$ supported on $\mathcal{R}_t$ that restricts to the reflection representation of the (nonparabolic) reflection subgroup $G(b,b,n)$.  Letting this representation $U_t$ play the role of the restricted reflection representation $V_t$, we again find a second representation $U'_t$ (using the reflection subgroup $G(1,1,n)$) that seems to play the role of $V^*_t$ in the generating function weighted by $\mathrm{det}(g)$.  In particular, both generating functions factor into linear factors over $\mathbb{Z}$, whose ``reflexponents'' are given in~\Cref{fig:non_restricted}.

\begin{example}
The group $G=G(6,2,3)$ has exponents $5,8,11$, has $648$ elements, and is not well-generated by ``simple reflections'' $s,t_2,t_2',t_3$.  The orbits of reflecting hyperplanes correspond to the partition $\{\{s\},\{t_2,t_2',t_3\}\}$.  There is a three-dimensional representation $U_t$ coming from the group $G(2,2,3)$ (the Weyl group of type $A_3\simeq D_3$) supported on the block $\{t_2,t_2',t_3\}$ with fake-degrees $3,6,9$:
\[s\mapsto \mathbbm{1},
t_2\mapsto \left(\begin{array}{ccc}1&0&0\\0&0&-1\\0&-1&0\end{array}\right),t_2'\mapsto \left(\begin{array}{ccc}1&0&0\\0&0&1\\0&1&0\end{array}\right),t_3\mapsto \left(\begin{array}{ccc}0&1&0\\1&0&0\\0&0&1\end{array}\right),\]
and a two-dimensional representation $U'_t$ with fake degrees $6,12$ coming from the group $G(1,1,3)$ (the Weyl group of type $A_2$):
\[s \mapsto \mathbbm{1}, t_2\mapsto \left(\begin{array}{ccc}-1&0\\0&1\end{array}\right), t_2' \mapsto \left(\begin{array}{ccc}-1&0\\0&1\end{array}\right), t_3 \mapsto \frac{1}{2}\left(\begin{array}{ccc}1&3\\1&-1\end{array}\right),\]

We confirm that 

\begin{align*}\sum_{g \in G} \left(x/y\right)^{\mathcal{M}_{U_t}(g)}y^{\mathcal{M}_V(g)}&=(1+9x+2y)(1+6x+2y)(1+3x+2y),\\
 \sum_{g \in G} \mathrm{det}(g)\left(x/y\right)^{\mathcal{M}_{U_t}(g)}y^{\mathcal{M}_V(g)} &= (1-12q - t)(1-6q - t)(1-t).
\end{align*}

\end{example}

\begin{figure}[htbp]
\[\begin{array}{cc}
G & t_i/t_{m-i}^* \\ \hline
G(ab,b,n) &\scalebox{0.8}{$\begin{array}{c} a(b{-}1),2a(b{-}1),\scalebox{0.5}{$\ldots$},(n{-}1)a(b{-}1),an \\(n{-}1)ab,\scalebox{0.5}{$\ldots$},2ab,ab,0 \end{array}$}\\ \hline
G_{13} & \begin{array}{c}4,8 \\ 12,0 \end{array}\\ \hline
\end{array}\]
\caption{Modified (co)reflexponents for the two not-well-restricted hyperplane orbits (filling in the ``$*$''s from \Cref{fig:table}).  One checks that \Cref{fact:hr_well_restricted} still holds for these numbers.}
\label{fig:non_restricted}
\end{figure}

T.~Douvropoulos has suggested that it should be possible to simultaneously apply Galois twists to the reflection representation and a well-restricted representation to obtain a common refinement of \Cref{thm:solomon2} and \cite[Theorem 3.3]{orlik1980unitary}.   He has also suggested generalizing the theorem in the style of~\cite[Theorem 2.3]{lehrer2003invariant}. 


\subsection{Extension of \Cref{thm:weighted} to dihedral groups}
We conclude with extensions of \Cref{thm:weighted}.  The only non-Weyl real reflection groups $W$ whose reflections form more than one conjugacy class are the dihedral group $I_2(2b)$ for $b\geq 4$.  These have two conjugacy classes of reflections $\mathcal{R}_s,\mathcal{R}_t$, each containing $b$ reflections, with corresponding positive roots $\Phi_s,\Phi_t$.  Weighting these by $q_1$ and $q_2$, it is easy to compute the corresponding generating function as~\cite{macdonald1972poincare} \begin{align*}\Ell_S(I_{2b};q_1,q_2)&:=\sum_{w \in I_2(2b)} q_1^{|\{\alpha \in \mathrm{inv}(w) \cap \Phi_1\}|} q_2^{|\{\alpha \in \mathrm{inv}(w) \cap \Phi_2\}|}\\&=(1+q_1)(1+q_2)\left(\frac{1-(q_1q_2)^{b}}{1-q_1q_2}\right).\end{align*}
Weighting reflections in $\mathcal{R}_s$ by $b$ and reflections in $\mathcal{R}_t$ by $1$ (so that $r=b$) gives a formula of the same form as the formula in~\Cref{thm:weighted} (now using reflexponents, rather than short exponents):
 \[\Ell_S(I_{2b};a,1)=\left(\frac{1+q^{b+1}+\cdots+q^{(b+1)(b-1)}}{1+q+\cdots+q^{b-1}}\right)\left(\frac{q^2-1}{q-1}\right)\left(\frac{q^{2b}-1}{q-1}\right),\]
since the degrees of $I_2(2b)$ are $d_1=2$ and $d_2=2b$, and the reflexponent is $b$.

\subsection{Possible generalizations of \Cref{thm:weighted} to complex reflection groups}
It would be desirable to find a graded space like the coinvariant algebra---perhaps by refining the polynomial invariants of the group, or by properly refining the cohomology of the (complexified) hyperplane complement---that recovers \Cref{thm:weighted} in the case of Weyl groups, but that also generalizes the theorem to complex reflection groups.  For example, as it restricts to a reflection representation, $V_\epsilon$ is \emph{amenable}---thus, the bigraded Poincar\'e series of \cite[Theorem 10.29]{lehrer2009unitary} for $(\mathbb{C}[V^*] \oplus \bigwedge V_\epsilon^*)^G$ seems somewhat related, perhaps by constructing an appropriate analogue of the coinvariant algebra.

From a more combinatorial perspective, there are some results towards finding ``length'' (or major index) statistics for complex reflection groups, so that the resulting generating function is equal to the Hilbert series for $\mathbb{C}[V^*]_{G}$ (see, for example,~\cite{bremke1997reduced} for $G(a,1,n)$).  It would be interesting to extend \Cref{thm:weighted} by modifying such statistics.

\section*{Acknowledgements}
I thank Vic Reiner and Christian Stump for helpful suggestions, the Banff Center for Arts and Creativity for excellent working conditions, and the organizers of the workshop ``Representation Theory: Connections to $(q,t)$-Combinatorics'' for inviting me.  I especially thank Theo Douvropoulos for many detailed comments and insights.  I thank Maxim Arnold and Carlos Arreche for inspiring conversations.  Calculations were done in Sage and GAP3 using the package CHEVIE~\cite{sagemath,Sch97,MR99m:20017}.  This work was partially supported by a Simons Foundation award.

\bibliographystyle{amsalpha}
\bibliography{shortexponents}

\end{document}